\documentclass{article}

\usepackage{PRIMEarxiv}

\usepackage[utf8]{inputenc} 
\usepackage[T1]{fontenc}    
\usepackage{hyperref}       
\usepackage{url}            
\usepackage{booktabs}       
\usepackage{amsfonts}       
\usepackage{nicefrac}       
\usepackage{microtype}      
\usepackage{lipsum}
\usepackage{fancyhdr}       
\usepackage{graphicx}       
\graphicspath{{media/}}     
\usepackage{fullpage}
\usepackage{graphicx}
\usepackage{multicol}
\usepackage[utf8]{inputenc}
\usepackage[english]{babel}
\usepackage{amsmath}
\usepackage{tikz}
\usepackage{pgfplots}
\usepackage{amssymb, graphics, amsmath}
\usepackage{multirow}
\usepackage{array}
\usepackage{bigints}
\usepackage{amsthm}
\usepackage{comment}
\usepackage{verbatim}
\usepackage{amsmath}

\newtheorem{theorem}{Theorem}[section]

\numberwithin{equation}{section}
\newtheorem{corollary}{Corollary}[section]
\newtheorem{remark}{Remark}[section]
\newtheorem*{theorem-non}{Theorem}
\newtheorem{fact}{Fact}[section]

\title{On the Law of Large Numbers and Convergence Rates for the Random Projections
}

\author{
  Vishakha \\
  Department of Mathematical Sciences \\
  University of Cincinnati \\
  Cincinnati\\
  \texttt{sharmav4@mail.uc.edu} \\
}

\begin{document}
\maketitle

\begin{abstract}
The aim of this paper is to establish the Marcinkiewicz-Zygmund (MZ) type law of large numbers for the randomly weighted sums with weights chosen randomly, uniformly over the unit sphere in $\mathbb{R}^n$. We also establish a theorem that describes the rate of convergence in the law of large numbers for these weighted sums.
\end{abstract}


\section{Introduction}

Let $(X_k)_{k \in \mathbb{Z}_+}$ be a sequence of real-valued random variables. For $n \in \mathbb{Z}_+$, let  $S_{n}(\Tilde{\theta}_n)$ denote the weighted sums
$$S_{n}(\Tilde{\theta}_n)=\theta_{n,1}X_{1}+\cdots+\theta_{n,n}X_{n},$$
where a random vector $\Tilde{\theta}_n=(\theta_{n,1},\dots,\theta_{n,n})$ is distributed uniformly on the unit sphere $\mathbb{S}^{n-1}$ of ${\mathbb{R}}^{n}$ ($n\geq2$), and 
is independent of $(X_1, \dots,X_n)$.
Let $P_{\Tilde{\theta}_n}$ 
be the unique rotationally
invariant probability measure on $\mathbb{S}^{n-1}$, referred to as the uniform distribution on the sphere. Throughout this paper, a random vector is said to be distributed uniformly on the sphere if it is distributed according to $P_{\Tilde{\theta}_n}$. \\

The above partial sums, $S_{n}(\Tilde{\theta}_n)$, can be viewed as one-dimensional random projections of a vector $(X_1, \dots,X_n)$ in the direction of $\Tilde{\theta}_n$. 
One of the primary applications of random projections is to reduce the dimensionality of data by projecting it onto a lower-dimensional space, thereby improving the efficiency of statistical computations, machine learning algorithms, and data visualization techniques.
The power of random projection comes from the Johnson-Lindenstrauss lemma \cite{MR737400} which guarantees that the distances between data points are approximately preserved even after being projected onto a random lower-dimensional space. This property of random projections is particularly useful for problems concerning the metric structure of data, such as clustering and nearest-neighbor search. 
Thus, random projections are very widely used and studied in the field of statistics and data analysis. This serves as the general motivation behind studying the weighted sums, $S_{n}(\Tilde{\theta}_n)$. \\

There has been significant work developed in this direction by Diaconis and Freedman \cite{MR751274}, Klartag and Sodin \cite{MR3136463}, Bobkov \emph{et.al.} 
 \cite{MR3858920} \cite{MR4112712}, Sudakov \cite{zbMATH03647821}, Dedecker \emph{et.al.} \cite{10.1214/24-AAP2121}, among others. 
Their work is primarily centered around investigating problems related to the distributions and deriving the estimates for the random projections, when the variables assume at least finite second moments. 
The natural question is then to ask whether we have law of large numbers for the random projections if moments lower than second moment are assumed. The goal of this paper is to bridge this gap. Under weaker conditions than identical distribution, we establish MZ type law of large numbers and $L^p$ convergence results when random variables have finite absolute $p$th moment, where $0< p < 2$. We also study the rate of convergence in the law of large numbers.\\

This paper is organized as follows. In Section \ref{S2}. we introduce notations and facts that are used throughout this paper. In Section \ref{S3}, we state our main results. Finally, in Section \ref{S4}, we present the proofs.

\section{Notations and Facts}\label{S2}

Let $(X_k)_{k \in \mathbb{Z}_+}$ be a sequence of real-valued random variables defined on a probability space $(\Omega,\mathcal{K},P_X)$.
Let $n\in \mathbb{Z}_+$ and $\xi_{n,1},\dots,\xi_{n,n}$ be i.i.d. standard normal random variables such that 
$\Tilde{\xi}_n=(\xi_{n,1},\dots,\xi_{n,n})$ is independent of $\Tilde{\theta}_n=(\theta_{n,1},\dots,\theta_{n,n})$ and $(X_1,\dots,X_n)$. Let $Z_{n}$ be a positive random variable such that $Z_{n}^{2}$ is chi-square with $n$ degrees of freedom, independent of the $\xi_{n,i}$'s, $\theta_{n,i}$'s, and $X_i$'s $(1\leq i \leq n)$. 
Observe that $\Tilde{\xi}_n$ has the same joint distribution as  $Z_n \Tilde{\theta}_n=(Z_n\theta_{n,1},\dots,Z_n\theta_{n,n})$ (see Appendix Fact \ref{R2}).\\

Throughout this paper, we construct all variables independent used on a corresponding probability space.
For ease of notation, for any random variables $W$ and $T$, we use $P_{W,T}$ for the product probability $P_W \times P_T$, and $E_{W,T}$ for the expectation w.r.t. the corresponding product space. Also, for a fixed $n$, $E_{\Tilde{\theta}_n}$ denotes the expectation w.r.t. $P_{\Tilde{\theta}_n}$.   We drop the subscripts wherever it is obvious.\\

Finally, $A \ll B$ means that $A\leq c B$ for some positive constant $c$ depending only on $p$ or $r$, which appear in Theorems \ref{T1}, \ref{T3}, and \ref{T4}.

\section{Main Results} \label{S3}

Let $(X_k)_{k \in \mathbb{Z}_+}$ be mean dominated by a random variable $Y$ in the sense that, for some $M>0$,  
\begin{equation}
\frac{1}{n}\sum_{k=1}^n P(|X_k|>x) \leq M P(|Y|>x), \quad \text{ for all } x>0 \text{ and } n\in \mathbb{Z}_+. \tag{MD}    
\end{equation}

Our first result establishes the convergence in probability and in $L^1$ for the random projections.

\begin{theorem}\label{T1}
Let  $0<p<2$. If $E|Y|^p < \infty$, then
$$n^{p/2-1} \cdot E_{\Tilde{\theta}_n} |S_n(\Tilde{\theta}_n)|^p \rightarrow 0 \quad \text{ in } P_X \text{ and in } L^1(\Omega) \quad \text{ as } \quad n \rightarrow \infty.$$ 
\end{theorem}

\begin{remark}\label{T1R2}
Under the assumptions of Theorem \ref{T1}, and an application of Markov's inequality implies that for every $\varepsilon>0$,
$$\lim_{n\rightarrow \infty} P_{X,\Tilde{\theta}_n} \left(|S_n(\Tilde{\theta}_n)| > \varepsilon n^{1/p-1/2}\right) =0.$$
\end{remark}

\noindent The following theorem describes the rate of convergence.

\begin{theorem}\label{T3}
Let $1\leq p<2$ and $1\leq r\leq p$. 
If $E|Y|^p< \infty$, then for every $\varepsilon >0$
$$\sum_{n=1}^{\infty}n^{p/r-2} \cdot P_{X}\left(\max_{1\leq k\leq n} k^{1/2}\cdot E_{\Tilde{\theta}_k}|S_k(\Tilde{\theta}_k)|>\varepsilon n^{1/r}\right) < \infty.$$   
\end{theorem}

\noindent As a consequence of this theorem, we achieve the following almost sure convergence.

\begin{corollary}\label{T3C1}
Under the assumptions of \textup{Theorem \ref{T3}}, 
$$n^{1/2-1/p} \cdot E_{\Tilde{\theta}_n}|S_n(\Tilde{\theta}_n)| \rightarrow 0 \quad P_{X}-\text{a.s.}  \quad \text{ as } \quad n \rightarrow \infty.$$   
\end{corollary}

\begin{theorem}\label{T4}
Let $1\leq p<2$ and $1\leq r < p$. 
If $E|Y|^p< \infty$, then for every $\varepsilon >0$
$$\sum_{n=1}^{\infty}n^{p/r-2} \cdot P_{X}\left(\max_{1\leq k\leq n} k^{r/2}\cdot E_{\Tilde{\theta}_k}|S_k(\Tilde{\theta}_k)|^r>\varepsilon n\right) < \infty.$$   
If in addition, $E|Y|^p\log^+|Y|< \infty$, then the above statement is true for $r=p$.
\end{theorem}

\begin{corollary}\label{T4C1}
Let $1< p<2$ and $E|Y|^p\log^+|Y|< \infty$, then
$$n^{p/2-1} \cdot E_{\Tilde{\theta}_n}|S_n(\Tilde{\theta}_n)|^p \rightarrow 0 \quad P_{X}-\text{a.s.}  \quad \text{ as } \quad n \rightarrow \infty.$$   
\end{corollary}

\section{Proofs}\label{S4}

\begin{proof}[Proof of \textup{Theorem \ref{T1}}]

We begin by using the fact that $(Z_{n}\theta_{n,1},\dots,Z_{n}\theta_{n,n})$ has the same distribution as $(\xi_{n,1},\dots,\xi_{n,n})$, and observe that
$$E_{\Tilde{\xi}_n}|S_n(\Tilde{\xi}_n)|^p = E_{\Tilde{\theta}_n,Z_n}|Z_n S_n(\Tilde{\theta}_n)|^p
=E_{Z_n}|Z_n|^p \cdot E_{\Tilde{\theta}_n}|S_n(\Tilde{\theta}_n)|^p,$$
where the last equality is due to the independence between $Z_n$ and $S_n(\Tilde{\theta}_n)$. 
Moreover, isolating $E_{\Tilde{\theta}_n}|S_n(\Tilde{\theta}_n)|^p$ in above equality, and using independence between the variables and $X$, we obtain
\begin{align*}
E_{X,\Tilde{\theta}_n}|S_n(\Tilde{\theta}_n)|^p 
=E^{-1}_{Z_n}|Z_n|^p \cdot E_{X,\Tilde{\xi}_n}|S_n(\Tilde{\xi}_n)|^p.  
\end{align*}   
Also, since for any positive integer $n$, 
$E_{Z_n}|Z_n|^p = \frac{2^{p/2}\cdot \Gamma(\frac{n+p}{2})}{\Gamma(\frac{n}{2})}$ (see Appendix Fact \ref{R1}), we get
\begin{align}
~& n^{p/2-1} \cdot E_{X,\Tilde{\theta}_n}|S_n(\Tilde{\theta}_n)|^p
=n^{p/2-1} \cdot \left(\frac{2^{p/2}\cdot \Gamma(\frac{n+p}{2})}{\Gamma(\frac{n}{2})}\right)^{-1} \cdot E_{X,\Tilde{\xi}_n}|S_n(\Tilde{\xi}_n)|^p \notag\\
&=n^{-1}  \cdot  \left(\frac{\Gamma(\frac{n+p}{2})}{(\frac{n}{2})^{p/2} \cdot \Gamma(\frac{n}{2})}\right)^{-1} \cdot E_{X,\Tilde{\xi}_n}|S_n(\Tilde{\xi}_n)|^p. \label{E1.1}
\end{align}
Note that, we have a classical asymptotic relation  due to Wendel \cite{MR29448},
\begin{align}
\frac{\Gamma(\frac{n+p}{2})}{(\frac{n}{2})^{p/2} \cdot \Gamma(\frac{n}{2})} \rightarrow 1 \quad \text{ as } \quad n\rightarrow \infty,  \label{E1.2}  
\end{align}
so, it is enough to show that $n^{-1} \cdot E_{X,\Tilde{\xi}_n}|S_n(\Tilde{\xi}_n)|^p \rightarrow 0$ as $n \rightarrow \infty$.\\\\
Let $\alpha \in (0,1/p-1/2)$. We truncate random variables $\{X_1,\dots,X_n\}$ at  level $n^{\alpha}$, and define
$$X'_k=X_{k}I(|X_{k}|\leq n^{\alpha}), \quad X''_k=X_{k}I(|X_{k}|>n^{\alpha}), \quad 1\leq k \leq n.$$
Furthermore, denote partial sums 
$$S'_n(\Tilde{\xi}_n)=\sum_{k=1}^n \xi_{n,k} X'_k \quad \text{and} \quad S''_n(\Tilde{\xi}_n)=\sum_{k=1}^n \xi_{n,k} X''_k.$$
Clearly, $S_n(\Tilde{\xi}_n)=S'_n(\Tilde{\xi}_n)+S''_n(\Tilde{\xi}_n)$.\\\\
Note that, for $0<p<2$, 
$E_{X,\Tilde{\xi}_n}^{1/p}|S'_n(\Tilde{\xi}_n)|^p \leq E_{X,\Tilde{\xi}_n}^{1/2}|S'_n(\Tilde{\xi}_n)|^2$. Also, 
because for $\{X'_k\}$'s fixed, $S_n'(\Tilde{\xi}_n)$ is normally distributed with variance $(X'_{1})^{2}+\cdots+(X'_{n})^{2}$, we obtain
\begin{align*}
E_{X,\Tilde{\xi}_n}|S'_n(\Tilde{\xi}_n)|^p 
\leq E_{X,\Tilde{\xi}_n}^{p/2}|S'_n(\Tilde{\xi}_n)|^2
=  \left(\sum_{k=1}^n E_{X}(X'_{k})^2\right)^{p/2}
\leq n^{p\alpha +p/2}.
\end{align*}
Thus, 
\begin{align}
n^{-1} \cdot E_{X,\Tilde{\xi}_n}^{}|S'_n(\Tilde{\xi}_n)|^p
\leq n^{p\alpha +p/2-1} \rightarrow 0 \quad \text{ as } \quad n \rightarrow \infty.\label{E1.3}
\end{align}\\
Moreover, since for $\{X''_k\}$'s fixed, $\{\xi_{n,1}X_1'',\dots,\xi_{n,n}X_n''\}$ are independent,  we obtain
\begin{align*}
E_{X,\Tilde{\xi}_n}|S''_n(\Tilde{\xi}_n)|^p  
&\ll 
E_X\left[\sum_{k=1}^n|X''_{k}|^p\right]
\ll
n\cdot E \left[|Y|^pI(|Y|>n^{\alpha})\right], 
\end{align*}
where the last inequality follows by using (MD) along with the fact  
$E_X |X''_{k}|^p = n^{p\alpha} P(|X_k|>n^{\alpha}) + \int_{n^{p \alpha}}^{\infty} P(|X_k|^p>x) dx$. Also, since $E|Y|^p$ is finite,   $E \left[|Y|^pI(|Y|>n^{\alpha})\right] \rightarrow 0$ as $n \rightarrow \infty$. Therefore, we get
\begin{align}
n^{-1} \cdot E_{X,\Tilde{\xi}_n}|S''_n(\Tilde{\xi}_n)|^p \rightarrow 0 \quad \text{ as } \quad n \rightarrow \infty. \label{E1.4}   
\end{align}\\
Then, by combining the results in (\ref{E1.3}) and (\ref{E1.4}), we obtain
\begin{align}
n^{-1} \cdot E_{X,\Tilde{\xi}_n}|S_n(\Tilde{\xi}_n)|^p  \rightarrow 0 \quad \text{ as } \quad n \rightarrow \infty. \label{E1.5}
\end{align}\\
Finally, (\ref{E1.5}) and (\ref{E1.2}) combined together with (\ref{E1.1}) give
\begin{align*}
n^{p/2-1} \cdot E_{X,\Tilde{\theta}_n}|S_n(\Tilde{\theta}_n)|^p  \rightarrow 0 \quad \text{ as } \quad n \rightarrow \infty. 
\end{align*}
\end{proof}

\begin{proof}[Proof of \textup{Theorem \ref{T3}}]
We use the arguments similar to those for equation (\ref{E1.1}) in Theorem \ref{T1}, and  notice that for any positive integer $k$,   
$$E_{X,\Tilde{\theta}_k}|S_k(\Tilde{\theta}_k)|
=\left(\frac{2^{1/2}\cdot \Gamma(\frac{k+1}{2})}{\Gamma(\frac{k}{2})}\right)^{-1} \cdot E_{X,\Tilde{\xi}_k}|S_k(\Tilde{\xi}_k)|.$$
This implies that 
\begin{align}
k^{1/2} \cdot E_{X,\Tilde{\theta}_k}|S_k(\Tilde{\theta}_k)|
= \frac{k^{1/2} \cdot \Gamma(\frac{k}{2})}{2^{1/2}\cdot \Gamma(\frac{k+1}{2})} \cdot E_{X,\Tilde{\xi}_k}|S_k(\Tilde{\xi}_k)|. \label{T3E1}
\end{align}
From equation (7) in  Wendel \cite{MR29448}, we have
\begin{align}
\left(\frac{x}{x+a}\right)^{1-a} \leq \frac{\Gamma(x+a)}{x^a \cdot \Gamma(x)} \leq 1, \quad \forall x, 0<a<1.  \label{W2}
\end{align}
Taking $x=k/2$ and $a=1/2$ in this equation and doing basic algebra, we  get the following bounds 
\begin{align}
1 \leq \frac{k^{1/2} \cdot \Gamma(\frac{k}{2})}{2^{1/2}\cdot \Gamma(\frac{k+1}{2})} 
\leq 2^{1/2}.\label{T3E2}
\end{align}
Using the facts we achieved in (\ref{T3E1}) and (\ref{T3E2}), we obtain
\begin{align*}
P_{X}\left(\max_{1\leq k\leq n} k^{1/2}\cdot E_{\Tilde{\theta}_k}|S_k(\Tilde{\theta}_k)|>\varepsilon n^{1/r}\right) 
\ll 
P_{X}\left(\max_{1\leq k\leq n} E_{\Tilde{\theta}_k}|S_k(\Tilde{\theta}_k)|>\varepsilon\cdot 2^{-1/2} \cdot n^{1/r}\right). 
\end{align*}
Therefore, the proof reduces to showing that for any $\varepsilon >0$ $\ $
\[
\sum_{n=1}^{\infty}n^{p/r-2} \cdot P_{X}\left(\max_{1\leq k\leq n} E_{\Tilde{\xi}_k}|S_k(\Tilde{\xi}_k)|>\varepsilon  n^{1/r}\right)
\]
is finite. \\\\
We split this series into two truncated parts and prove that each part is finite.
We use truncation of variables $\{X_1,\dots,X_n\}$ at  level $n^{1/r}$, and define
$$X'_k=X_{k}I(|X_{k}|\leq n^{1/r}), \quad X''_k=X_{k}I(|X_{k}|>n^{1/r}), \quad 1\leq k \leq n.$$
Also, let 
$S'_n(\Tilde{\xi}_n)=\sum_{k=1}^n \xi_{n,k} X'_k$ and $S''_n(\Tilde{\xi}_n)=\sum_{k=1}^n \xi_{n,k} X''_k.$
Clearly, $S_n(\Tilde{\xi}_n)=S'_n(\Tilde{\xi}_n)+S''_n(\Tilde{\xi}_n)$.\\\\
H\"older's inequality gives
$E_{\Tilde{\xi}_k}|S'_k(\Tilde{\xi}_k)| \leq E_{\Tilde{\xi}_k}^{1/2}|S'_k(\Tilde{\xi}_k)|^2$. This is further bounded by $(\sum_{j=1}^{k} |X_j'|^2)^{1/2}$, because $S'_k(\Tilde{\xi}_k)$ is normally distributed with variance $\sum_{j=1}^{k} |X_j'|^2$ for fixed $\{X'_j\}$'s. These observations followed by Markov's inequality gives
\begin{align}
~&\sum_{n=1}^{\infty}n^{p/r-2} \cdot P_{X}\left(\max_{1\leq k\leq n} E_{\Tilde{\xi}_k}|S'_k(\Tilde{\xi}_k)|>\varepsilon  n^{1/r}\right)  \notag \\
&\leq 
\sum_{n=1}^{\infty}n^{p/r-2} \cdot P_{X}\left(\max_{1\leq k\leq n} \sqrt{\sum_{j=1}^{k} |X_j'|^2}>\varepsilon  n^{1/r}\right) \notag \\
&\leq 
\sum_{n=1}^{\infty}n^{p/r-2} \cdot P_{X}\left(\sum_{k=1}^{n} |X_k'|^2>\varepsilon^2  n^{2/r}\right) 
\ll
\sum_{n=1}^{\infty}n^{p/r-2} \cdot  n^{-2/r} \cdot \sum_{k=1}^{n} E_X|X_k'|^2. \label{T3E4}
\end{align}
Also,  by using triangle's inequality and  Markov's inequality, we have
\begin{align}
~&\sum_{n=1}^{\infty}n^{p/r-2} \cdot P_{X}\left(\max_{1\leq k\leq n} E_{\Tilde{\xi}_k}|S''_k(\Tilde{\xi}_k)|>\varepsilon  n^{1/r}\right)  \notag \\
&\ll
\sum_{n=1}^{\infty}n^{p/r-2} \cdot P_{X}\left(\sum_{k=1}^{n} |X_k''| > \varepsilon  n^{1/r}\right) \ll
\sum_{n=1}^{\infty} n^{p/r-2} \cdot n^{-1/r} \cdot \sum_{k=1}^n E_X |X_k''|. \label{T3E5}
\end{align}
Since $E|Y|^p$ is finite and we have (MD), both the series in (\ref{T3E4}) and (\ref{T3E5}) are finite due to Lemma 3.2 \cite{MR3567927}. \\\\
Thus, 
\begin{align*}
\sum_{n=1}^{\infty}n^{p/r-2} \cdot P_{X}\left(\max_{1\leq k\leq n} E_{\Tilde{\xi}_k}|S_k(\Tilde{\xi}_k)|>\varepsilon  n^{1/r}\right) < \infty,   
\end{align*}
which completes the proof.\\
\end{proof}

\begin{proof}[Proof of \textup{Corollary \ref{T3C1}}]
Letting $r=p$ in Theorem \ref{T3}, we notice that for any $\varepsilon>0$,
\begin{align}
\sum_{n=1}^{\infty}n^{-1} \cdot P_{X}\left(\max_{1\leq k\leq n} k^{1/2} \cdot E_{\Tilde{\theta}_k}|S_k(\Tilde{\theta}_k)|>\varepsilon  n^{1/p}\right) < \infty. \label{T3C1E1}
\end{align}
Observe that for any $\delta>0$ and $k\in \mathbb{Z}_+$,
\begin{align*}
P_{X}\left(\underset{m \geq k}{\sup}
\frac{E_{\Tilde{\theta}_m}|S_m(\Tilde{\theta}_m)|}{m^{1/p-1/2}}>\delta \right)
\leq 
\sum_{i>\log_2 k}P_{X}\left(\underset{1\leq m\leq 2^{i}}{\max} m^{1/2} \cdot E_{\Tilde{\theta}_m}|S_m(\Tilde{\theta}_m)| > 2^{(i-1)/p} \cdot \delta\right) \\
\leq 
\frac{1}{2}
\sum_{n>k}n^{-1} \cdot P_{X}\left(\underset{1 \leq m \leq n}{\max} m^{1/2} \cdot E_{\Tilde{\theta}_m}|S_m(\Tilde{\theta}_m)|> 2^{-2/p}\cdot  \delta \cdot n^{1/p}\right),
\end{align*}
where the last expression approaches $0$ as $k \rightarrow \infty$ by using (\ref{T3C1E1}). \\\\
Therefore, $P_{X}\left(\underset{m \geq k}{\sup}
\frac{E_{\Tilde{\theta}_m}|S_m(\Tilde{\theta}_m)|}{m^{1/p-1/2}}>\delta \right) \rightarrow 0$ as $k \rightarrow \infty$. This further implies that 
$$n^{1/2-1/p} \cdot E_{\Tilde{\theta}_n}|S_n(\Tilde{\theta}_n)| \xrightarrow[]{P_{X}-\text{a.s.}} 0 \quad \text{as } \quad  n \rightarrow \infty.$$
\end{proof}

\begin{proof}[Proof of \textup{Theorem \ref{T4}}]
The proof is essentially similar to that of Theorem \ref{T3}. So, most of the details will be skipped.\\\\
We start with the relation 
$$E_{X,\Tilde{\theta}_k}|S_k(\Tilde{\theta}_k)|^r
=\left(\frac{2^{r/2}\cdot \Gamma(\frac{k+r}{2})}{\Gamma(\frac{k}{2})}\right)^{-1} \cdot E_{X,\Tilde{\xi}_k}|S_k(\Tilde{\xi}_k)|^r.$$
Taking $x=k/2$ and $a=r/2$ in equation (\ref{W2}) and doing basic algebra, we  get the bounds
\begin{align*}
1 \leq \frac{k^{r/2} \cdot \Gamma(\frac{k}{2})}{2^{r/2}\cdot \Gamma(\frac{k+r}{2})} 
\leq (1+r)^{1-r/2}.
\end{align*}
Using the relation and the bounds we obtained above, we get
\begin{align*}
P_{X}\left(\max_{1\leq k\leq n} k^{r/2}\cdot E_{\Tilde{\theta}_k}|S_k(\Tilde{\theta}_k)|^r >\varepsilon n\right) 
\ll 
P_{X}\left(\max_{1\leq k\leq n} E_{\Tilde{\theta}_k}|S_k(\Tilde{\theta}_k)|^r \gg \varepsilon  n \right). 
\end{align*}
Therefore, we just need to show that for any $\varepsilon >0$ $\ $
\[
\sum_{n=1}^{\infty}n^{p/r-2} \cdot P_{X}\left(\max_{1\leq k\leq n} E_{\Tilde{\xi}_k}|S_k(\Tilde{\xi}_k)|^r>\varepsilon  n \right)
\]
is finite. \\\\
We use the same truncations as we did in Theorem \ref{T3}. \\\\
For the first part, we use H\"older's inequality to get
$E_{\Tilde{\xi}_k}|S'_k(\Tilde{\xi}_k)|^r \leq E_{\Tilde{\xi}_k}^{r/2}|S'_k(\Tilde{\xi}_k)|^2$, and then follow the series of arguments from Theorem \ref{T3} to obtain
\begin{align}
~&\sum_{n=1}^{\infty} n^{p/r-2} \cdot P_{X}\left(\max_{1\leq k\leq n} E_{\Tilde{\xi}_k}|S'_k(\Tilde{\xi}_k)|^r>\varepsilon  n \right)  
\ll
\sum_{n=1}^{\infty} n^{(p-2)/r-2}  \cdot \sum_{k=1}^{n} E_X|X_k'|^2. \label{T4E1}
\end{align}\\
For the second part, an application of $r$th moment inequality for the independent variables  $\{\xi_{n,k} X''_k\}$'s (for fixed $\{ X''_k\}$'s), and the Markov's inequality gives
\begin{align}
~&\sum_{n=1}^{\infty}n^{p/r-2} \cdot P_{X}\left(\max_{1\leq k\leq n} E_{\Tilde{\xi}_k}|S''_k(\Tilde{\xi}_k)|^r>\varepsilon  n\right)  \notag \\
&\ll
\sum_{n=1}^{\infty} n^{p/r-3} \sum_{k=1}^n E_X |X_k''|^r 
\ll
\sum_{n=1}^{\infty} n^{p/r-2} \cdot E|Y|^rI(|Y|>n^{1/r}),  \label{T4E2}
\end{align}
the last inequality is due to (MD). \\\\
Furthermore, it can be easily shown (by carefully using the arguments similar to Lemma 3.2 \cite{MR3567927}) that the last series in (\ref{T4E2}) is finite by our assumptions $E|Y|^p < \infty$ for $r<p$, and $E|Y|^p\log^+|Y| < \infty$ for $r=p$.\\\\
Therefore, we obtain 
\begin{align*}
\sum_{n=1}^{\infty}n^{p/r-2} \cdot P_{X}\left(\max_{1\leq k\leq n} E_{\Tilde{\xi}_k}|S_k(\Tilde{\xi}_k)|^r > \varepsilon  n \right) < \infty.   
\end{align*}
This completes the proof.\\
\end{proof}

\begin{proof}[Proof of \textup{Corollary \ref{T4C1}}]
We provide a brief proof, as it mirrors the argument of the previous Corollary \ref{T3C1}.\\\\
Since $E|Y|^p\log^+|Y|< \infty$, taking $r=p$ in Theorem \ref{T4} implies that for any $\varepsilon>0$,
\begin{align}
\sum_{n=1}^{\infty}n^{-1} \cdot P_{X}\left(\max_{1\leq k\leq n} k^{p/2} \cdot E_{\Tilde{\theta}_k}|S_k(\Tilde{\theta}_k)|^p>\varepsilon  n\right) < \infty. \label{T4C1E1}
\end{align}
Then, for any $\delta>0$ and $k\in \mathbb{Z}_+$,
\begin{align*}
P_{X}\left(\underset{m \geq k}{\sup}
\frac{E_{\Tilde{\theta}_m}|S_m(\Tilde{\theta}_m)|^p}{m^{1-p/2}}>\delta \right)
\leq 
\frac{1}{2}
\sum_{n>k}n^{-1} \cdot P_{X}\left(\underset{1 \leq m \leq n}{\max} m^{p/2} \cdot E_{\Tilde{\theta}_m}|S_m(\Tilde{\theta}_m)|^p> \delta n/4\right),
\end{align*}
where the last expression approaches $0$ as $k \rightarrow \infty$ by using (\ref{T4C1E1}). \\\\
This implies that
$$n^{p/2-1} \cdot E_{\Tilde{\theta}_n}|S_n(\Tilde{\theta}_n)|^p \xrightarrow[]{P_{X}-\text{a.s.}} 0 \quad \text{as } \quad  n \rightarrow \infty.$$
\end{proof}


\section{Appendix}
\begin{fact}\label{R2}
Let $\Tilde{\theta}_n$ and $\Tilde{\xi}_n$ be random vectors as defined in Section \ref{S2}.  
Let $Z_{n}$ be a positive random variable such that $Z_{n}^{2} \sim \chi_n^2$, independent of  $\Tilde{\xi}_n$ and $\Tilde{\theta}_n$. 
Then 
$$\Tilde{\xi}_n \overset{d}{=} Z_n \Tilde{\theta}_n.$$
\end{fact}

\begin{proof}
Since $\Tilde{\theta}_n \sim$ Unif($\mathbb{S}^{n-1}$), its characteristic function $E_{\Tilde{\theta}_n}(e^{i\langle t,\Tilde{\theta}_n \rangle}) = h(\Vert t \Vert)$, for some function $h(\cdot)$. 
It can also  be noticed that for any $t\in \mathbb{R}^n$, 
$E_{\Tilde{\xi}_n}(e^{i\langle t,\Tilde{\xi}_n \rangle}) = E_{\Tilde{\xi}_n}(e^{i\langle \Vert t \Vert u, \Tilde{\xi}_n \rangle})$, $\forall u \in \mathbb{S}^{n-1}$. Due to these observations, we have 
\begin{align*}
E_{\Tilde{\theta}_n} E_{\Tilde{\xi}_n}(e^{i\langle t,\Tilde{\xi}_n \rangle}) 
= E_{\Tilde{\theta}_n} E_{\Tilde{\xi}_n}(e^{i\langle \Vert t \Vert \Tilde{\theta}_n, \Tilde{\xi}_n \rangle})
= E_{\Tilde{\xi}_n} (h(\Vert t \Vert \Vert \Tilde{\xi}_n \Vert)),
\end{align*}
where $E_{\Tilde{\xi}_n} (h(\Vert t \Vert \Vert \Tilde{\xi}_n \Vert)) = E_{Z_n} (h(\Vert t \Vert Z_n))$, because  $\Vert \Tilde{\xi}_n \Vert^2  \overset{d}{=} Z_{n}^{2}$. Thus, we have
\begin{align*}
E_{\Tilde{\xi}_n}(e^{i\langle t,\Tilde{\xi}_n \rangle}) 
= E_{Z_n} (h(\Vert t \Vert Z_n))
= E_{Z_n} E_{\Tilde{\theta}_n}(e^{i\langle t Z_n, \Tilde{\theta}_n \rangle})
=E_{Z_n,\Tilde{\theta}_n}(e^{i\langle t, Z_n \Tilde{\theta}_n \rangle}).
\end{align*}
This concludes the proof.\\
\end{proof}

\begin{fact}\label{R1}
Let $n\in \mathbb{Z}_+$ and $s>0$. If $Z\sim \chi_n^2$, then $EZ^s= \frac{2^s \cdot \Gamma(\frac{n+2s}{2})}{\Gamma{(\frac{n}{2})}}$.
\end{fact}

\begin{proof}
The pdf of $\chi_n^2$ is given by
\begin{equation*}
  f_{\chi_n^2}(x;n) =
    \begin{cases}
     \frac{x^{n/2-1} \cdot e^{-x/2}}{2^{n/2}\cdot \Gamma(n/2)},  & x>0;\\
     0,  & \text{otherwise.}
    \end{cases}       
\end{equation*}
Then,
\begin{align*}
EZ^s
= \int_{0}^{\infty} x^s \cdot \frac{x^{n/2-1} \cdot e^{-x/2}}{2^{n/2}\cdot \Gamma(n/2)} ~dx 
&= \frac{2^s\cdot \Gamma(\frac{n+2s}{2})}{\Gamma{(\frac{n}{2})}} \int_{0}^{\infty}  \frac{x^{(n+2s)/2-1}\cdot e^{-x/2}}{2^{(n+2s)/2}\cdot \Gamma(\frac{n+2s}{2})} ~dx \\
=\frac{2^s\cdot \Gamma(\frac{n+2s}{2})}{\Gamma{(\frac{n}{2})}}.
\end{align*}
\end{proof}

\section*{Acknowledgments}
The research was supported by the NSF Grant No.  DMS-2054598.
The author is grateful to Magda Peligrad for providing the insightful guidance in the completion of this work.

\bibliographystyle{amsplain}
\bibliography{ref2}

\providecommand{\bysame}{\leavevmode\hbox to3em{\hrulefill}\thinspace}
\providecommand{\MR}{\relax\ifhmode\unskip\space\fi MR }
\providecommand{\MRhref}[2]{%
  \href{http://www.ams.org/mathscinet-getitem?mr=#1}{#2}
}
\providecommand{\href}[2]{#2}
\begin{thebibliography}{1}

\bibitem{MR3858920}
S.~G. Bobkov, G.~P. Chistyakov, and F.~G\"otze, \emph{Berry-{E}sseen bounds for typical weighted sums}, Electron. J. Probab. \textbf{23} (2018), Paper No. 92, 22. \MR{3858920}

\bibitem{MR4112712}
\bysame, \emph{Normal approximation for weighted sums under a second-order correlation condition}, Ann. Probab. \textbf{48} (2020), no.~3, 1202--1219. \MR{4112712}

\bibitem{10.1214/24-AAP2121}
J{\'e}r{\^o}me Dedecker, Florence Merlev{\`e}de, and Magda Peligrad, \emph{{Rates in the central limit theorem for random projections of martingales}}, The Annals of Applied Probability \textbf{35} (2025), no.~1, 564 -- 589.

\bibitem{MR751274}
Persi Diaconis and David Freedman, \emph{Asymptotics of graphical projection pursuit}, Ann. Statist. \textbf{12} (1984), no.~3, 793--815. \MR{751274}

\bibitem{MR737400}
William~B. Johnson and Joram Lindenstrauss, \emph{Extensions of {L}ipschitz mappings into a {H}ilbert space}, Conference in modern analysis and probability ({N}ew {H}aven, {C}onn., 1982), Contemp. Math., vol.~26, Amer. Math. Soc., Providence, RI, 1984, pp.~189--206. \MR{737400}

\bibitem{MR3136463}
B.~Klartag and S.~Sodin, \emph{Variations on the {B}erry-{E}sseen theorem}, Teor. Veroyatn. Primen. \textbf{56} (2011), no.~3, 514--533. \MR{3136463}

\bibitem{zbMATH03647821}
V.~N. Sudakov, \emph{Typical distributions of linear functionals in finite-dimensional spaces of higher dimension}, Sov. Math., Dokl. \textbf{19} (1978), 1578--1582 (English).

\bibitem{MR29448}
J.~G. Wendel, \emph{Note on the gamma function}, Amer. Math. Monthly \textbf{55} (1948), 563--564. \MR{29448}

\bibitem{MR3567927}
Na~Zhang, \emph{On the law of large numbers for discrete {F}ourier transform}, Statist. Probab. Lett. \textbf{120} (2017), 101--107. \MR{3567927}

\end{thebibliography}

\end{document}